\documentclass[11pt]{article}
\usepackage{amsmath,amsfonts,amsthm}
\usepackage{ifthen}
\usepackage{amssymb}
\usepackage{enumitem}

\newtheorem{theorem}{Theorem}
\newtheorem{lemma}[theorem]{Lemma}

\newtheorem{proposition}[theorem]{Proposition}

\begin{document}

\title{Spectrum of Cayley graphs on the symmetric group generated by transpositions}

\author{
   Roi Krakovski\thanks{Supported in part by postdoctoral support at the Simon Fraser University.}\\[1mm]
   {Department of Mathematics}\\{Simon Fraser University}\\{Burnaby, B.C. V5A 1S6}
\and
   Bojan Mohar\thanks{Supported in part by an NSERC Discovery Grant (Canada), by the Canada Research Chair program, and by the Research Grant P1--0297 of ARRS (Slovenia).}~~\thanks{On leave from: IMFM \& FMF, Department of Mathematics, University of Ljubljana, Ljubljana, Slovenia.}\\[1mm]
   {Department of Mathematics}\\{Simon Fraser University}\\{Burnaby, B.C. V5A 1S6}
   }


\maketitle

\begin{abstract}
For an integer $n\geq 2$, let $X_n$ be the Cayley graph on the symmetric group $S_n$ generated by the set of transpositions $\{(1~2),(1~3),\dots,(1~n)\}$. It is shown that the spectrum of $X_n$ contains all integers from $-(n-1)$ to $n-1$ (except 0 if $n=2$ or $n=3$).
\end{abstract}

\section{Introduction}
\label{intro}

For a graph $X$, the \emph{spectrum} of $X$ is the spectrum of its adjacency matrix.
For a group $G$ and a subset $T\subseteq G$, the \emph{Cayley graph} of $G$ generated by $T$, denoted by $X(G,T)$, is the graph with vertex set $G$ and an edge $(g,tg)$ for each $g\in G$ and $t\in T$. If $T$ is inverse-closed, then $X(G,T)$ is an undirected graph, which we will assume henceforth.

In this paper we are interested in the spectrum of the Cayley graph $X(S_n,T_n)$, where $n\geq2$ is an integer, $S_n$ is the symmetric group and $T_n:=\{(1~2),(1~3),\dots,(1~n)\} \subseteq S_n$. Friedman~\cite{Friedman00} proved that if $T_n' \subseteq S_n$ is any set of $n-1$ transpositions and $T_n' \neq T_n$ (up to conjugacy) then the spectrum of $X(S_n,T_n')$ is never integral. Abdollahi and Vatandoost~\cite{AV} conjectured that spectrum of  $X(S_n,T_n)$ is integral, and contains all integers in the range from $-(n-1)$ to $n-1$ (with the sole exception that when $n\leq3$, zero is not an eigenvalue of $X(S_n,T_n)$). Using a computer, the conjecture was verified for $n\leq 6$.
In this paper we prove the second part of the conjecture.

\begin{theorem}
\label{main}
Let $n \geq 2$ be an integer and let $T_n=\{(1~2),(1~3),\dots,(1~n)\} \subseteq S_n$. For each integer $1 \leq \ell \leq n-1$,  $\pm(n-\ell)$ are eigenvalues of $X(S_n,T_n)$ with multiplicity at least ${n-2 \choose \ell-1}$. If $n\ge4$, then $0$ is an eigenvalue of $X(S_n,T_n)$ with multiplicity at least ${n-1 \choose 2}$.
\end{theorem}

Note that $\pm(n-1)$ is a simple eigenvalue of $X(S_n,T_n)$ since the graph is $(n-1)$-regular, bipartite, and connected.

\section{Partial permutation graphs}
\label{sec:PartialPermutationGraphs}

In this section we introduce a family of graphs, called \emph{partial permutation graphs}.
Let $d$ and $n$ be positive integers such that $1 \leq d \leq n$. Let $S_{d,n}$ be the set of all $d$-tuples with entries from the set $[n]=\{1,\dots,n\}$ having no repeated entries, that is,
$$
   S_{d,n} = \{(a_1,\dots,a_d)\mid a_1,\dots,a_d \in [n]~\textrm{ and~} a_i\neq a_j \textrm{~for~} 1\leq i<j \leq d\}.
$$
The \emph{partial permutation graph} $P(d,n)$ is defined as follows. Its vertex set is $S_{d,n}$, and two $d$-tuples are adjacent if and only if they differ in exactly one coordinate. The following lemma, whose proof is straightforward, lists some basic properties of $P(d,n)$.

\begin{lemma}
\label{basic-prop}
Let $d$ and $n$ be positive integers such that $1 \leq d \leq n$. Then $P(d,n)$ satisfies the following properties:
\begin{enumerate}
	\item[\rm(i)] $|V(P(d,n))|=\frac{n!}{(n-d)!}$.
	\item[\rm(ii)] The graph $P(d,n)$ is $d(n-d)$-regular.
	\item[\rm(iii)] The cardinality of every maximal clique in $P(d,n)$ is $n-d+1$.
\end{enumerate}
\end{lemma}

Next we show that $X(S_n,T_n)$ is a partial permutation graph.

\begin{lemma}
\label{bipartite}
Let $n\geq2$ be an integer. Then $P(n-1,n)$ is isomorphic to $X(S_n,T_n)$. In particular, $P(n-1,n)$ is bipartite.
\end{lemma}

\begin{proof}
Set $P:=P(n-1,n)$ and $X:=X(S_n,T_n)$. For an $n$-tuple $(c_1,\dots,c_n)$ of pairwise distinct elements of $[n]$, let $\pi_{(c_1,\dots,c_n)} \in S_n$ be the permutation defined by $\pi_{(c_1,\dots,c_n)}(c_j)=j$ ($j=1,\dots,n$).

Let $\phi:V(P) \rightarrow V(X)$ be the mapping defined by $\phi((a_1,\dots,a_{n-1}))=\pi_{(a_0,a_1,\dots,a_{n-1})}$, where $\{a_0\} = [n] \setminus \{ a_1,\dots,a_{n-1} \}$. We claim that $\phi$ is a graph isomorphism. Clearly, $\phi$ is bijective. It remains to show that $p_1p_2 \in E(P)$ if and only if $\phi(p_1)\phi(p_2) \in E(X)$.

Suppose $p_1p_2 \in E(P)$. Let $p_1=(a_1,\dots,a_{n-1})$ and $p_2=(b_1,\dots,b_{n-1})$. Let $1 \leq \ell \leq n-1$ so that $p_1$ and $p_2$ differ  in the $\ell$th  coordinate (since $p_1p_2 \in E(P)$, $\ell$ exists and it is unique). By definition of $\phi$, we have $\phi(p_1)=(b_{\ell},a_1,\dots, a_{n-1})$ and $\phi(p_2)=(a_{\ell},a_1,\dots,a_{\ell-1},b_{\ell},a_{\ell+1},\dots, a_{n-1})$. Then $\phi(p_1)=\sigma \cdot \pi_{(a_{\ell},a_1,\dots,a_{\ell-1},b_{\ell},a_{\ell+1},\dots, a_{n-1})} = \sigma\cdot\phi(p_2)$, where $\sigma \in S_n$ is the transposition $(1~\ell+1)$, hence $\phi(p_1)\phi(p_2) \in E(X)$.

For the converse, observe that if $p_1p_2 \notin E(P)$, then $\phi(p_1)$ and $\phi(p_2)$ differ in at least two coordinates that are different from the first coordinate. Hence, $\phi(p_1)\neq \sigma \cdot \phi(p_2)$ for any transpositions $\sigma$ of the form $(1~i)$ ($1 < i \leq n$).
\end{proof}

For integers $1\leq d<n$, let $I(n,d)$ be the set of all subsets of $[n]$ of cardinality $d+1$. For $I \in I(n,d)$, let $A_I$  be the set of all $d$-tuples (with no repetitions) with entries from the set $I$. Note that $|I(n,d)|={n \choose d+1}$ and $|A_I|= (d+1)!$.  Let $\mathcal{A}(n,d):=\{A_I\mid I \in I(n,d)\}$.

For $\mathcal{A}\subseteq \mathcal{A}(n,d)$, we say that  a $d$-tuple $t$ is \emph{unique} with respect to $\mathcal{A}$, if there exists $A \in \mathcal{A}$, such that $t \in A$ but $t \notin A'$, for every $A' \in \mathcal{A} \setminus A$. We say that $\mathcal{A}$ is \emph{independent}, if every subset of $\mathcal{A}$ has (at least one) unique $d$-tuple.  With this notation we have the following.

\begin{lemma}
\label{IS}
There exists an independent set $\mathcal{A} \subseteq \mathcal{A}(n,d)$ of cardinality $n-1 \choose d$.
\end{lemma}

\begin{proof}
We proceed by induction on $n+d$. If $d=1$, let $\mathcal{A}=\{A_{\{1,2\}},A_{\{1,3\}},\dots,
\allowbreak A_{\{1,n\}}\}$. Note that $\{1,i\}\in I(n,1)$ and $A_{\{1,i\}}=\{(1),(i)\}$, for $i=2,\dots,n$. Clearly, $\mathcal{A}$ is an independent set of cardinality $n-1$ and the claim follows. If $d=n-1$, then $|I(n,n-1)|=1$, and the claim follows with $\mathcal{A}=\{A_{\{1,\dots,n\}}\}$.

We may assume that $d \geq 2$ and $n\geq d+2$. Let $\mathcal{A}^*(n,d)= \mathcal{A}(n,d) \setminus \mathcal{A}(n-1,d)$.
By the induction hypothesis, there exists an independent set $\mathcal{A}_1 \subseteq \mathcal{A}(n-1,d)$ of cardinality $n-2 \choose d$. Next we claim that

\begin{itemize}
	\item [(1)] there exists an independent set $\mathcal{A}_2 \subseteq \mathcal{A}^*(n,d)$ of cardinality $n-2 \choose d-1$, such that  for every subset $\mathcal{A}_2' \subseteq \mathcal{A}_2$ there is a unique $d$-tuple with respect to $\mathcal{A}_2'$, such that one of its entries contains the value $n$.
\end{itemize}
To prove (1), let  $\phi: \mathcal{A}^*(n,d) \rightarrow  \mathcal{A}(n-1,d-1)$ be a mapping defined by $$\phi(A_I)=A_{I \setminus \{n\}}$$ where $ I \in I(n,d)$ and $n \in I$.

Clearly, $\phi$ is bijective.  By the induction hypothesis, there is an independent set $\mathcal{A}' \subseteq \mathcal{A}(n-1,d-1)$ of cardinality $n-2 \choose d-1$. It is easily seen that $\mathcal{A}_2:=\phi^{-1}(\mathcal{A'})$ is an independent set in $\mathcal{A}^*(n,d)$ with desired property. This proves (1).

\smallskip

By definition, $\mathcal{A}_1$ and $\mathcal{A}_2$ are disjoint, hence $|\mathcal{A}|=|\mathcal{A}_1|+|\mathcal{A}_2|={n-2 \choose d} + {n-2 \choose d-1}= {n-1 \choose d}$.
To conclude the proof it remains to show that $\mathcal{A}=\mathcal{A}_1 \cup \mathcal{A}_2 \subseteq \mathcal{A}(n,d)$ is an independent set.

Let $\mathcal{A}' \subseteq \mathcal{A}$. We will show that $\mathcal{A}'$ has a unique $d$-tuple. This is trivially true if $\mathcal{A}' \subseteq \mathcal{A}_1$. Hence, we may assume that $ \mathcal{A}_2 \cap \mathcal{A'} \neq \emptyset$. Let $t$ be the unique $d$-tuple with respect to $\mathcal{A}_2 \cap \mathcal{A'} \subseteq \mathcal{A}_2$ as exists by (1). Then, one of the entries of $t$ contains the value $n$. The uniqueness of $t$  with respect to $\mathcal{A}'$ follows since $t$ is unique with respect to $\mathcal{A}_2$ and for each $A \in \mathcal{A}_1$, no $d$-tuple of $A$ has an entry with value $n$. This concludes the proof.
\end{proof}

\begin{lemma}
\label{function}
Let\/ $1\leq d < n $ be  integers and let $X=P(d,n)$. Then there exists a family  $\mathcal{F}_X$ of functions, each with domain $V(X)$ and range $\{-1,0,1\}$, such that the following holds:
\begin{enumerate}
	\item[\rm(i)] $\{-1,1\} \subseteq Im(\phi)$, for every $\phi \in \mathcal{F}_X$.
	\item[\rm(ii)] $\sum_{v \in V(K)} \phi(v) =0$, for every $\phi \in \mathcal{F}_X$ and every maximal clique $K$.
	\item[\rm(iii)] If we view each $\phi$ as a vector in $\mathbb{R}^{V(X)}$, then $\mathcal{F}_X$ contains a linearly independent set of cardinality $n-1 \choose d$.
\end{enumerate}
\end{lemma}

\begin{proof}
Consider $A_I \in \mathcal{A}(n,d)$, for some $I \in I(n,d)$. Since each element of $A_I$ is a $d$-tuple, we can view $A_I$ as a subset  of $V(X)$. Then, $X[A_I]\subseteq X$ (the subgraph of $X$ induced on the vertex set $A_I$) is isomorphic to $P(d,d+1)$. By Lemma \ref{bipartite}, $X[A_I]$ is bipartite. Let $A_I^1,A_I^2 \subseteq A_I$ ($A_I^1 \cup A_I^2=A_I$) be the corresponding bipartition. Let $\phi_{I}:V(X) \rightarrow \{-1,0,1\}$ be defined by
$$
\phi_{I}(v) = \left\{
              \begin{array}{rl}
              	-1, & v \in A_I^1,\\[1mm]
                1, & v \in A_I^2,\\[1mm]
                0, & v\in V(X)\setminus A_I.
              \end{array}
             \right.
$$

Let $\mathcal{F}_X:=\{\phi_{I}\mid I \in I(n,d)\}$. We will show that $\mathcal{F}_X$ satisfies (i)--(iii). Property (i) is satisfied trivially, since for every  $I \in I(n,d)$, $X[A_I]$ is bipartite with at least one edge.

For (ii) we argue as follows.
We may assume that $n\geq d+2$, for if $n=d+1$ then $X[A_I]$ is isomorphic to $X$ (since then $A_I=V(X)$) and $\phi_{I}$ satisfies (ii) as required.
Let $K \subseteq V(X)$ be a maximal clique in $X$. Since $n \geq d+2$, every edge of $X$ is contained in a clique of size at least 3, and so $|K| \geq 3$.  Since $X[A_I]$ is bipartite, $|K \cap A_I| \leq 2$. We claim that
\begin{enumerate}
	\item [(1)] $|K\cap A_I| \neq 1$.
\end{enumerate}
To prove (1), suppose to the contrary that $|K\cap A_I|=1$. Let $v \in K\cap A_I$. Let $a \in I$ be such that $a$ does not appear in the $d$-tuple $v$ ($a$ exists since $|I|=d+1$). Since $K$ is a clique, the $d$-tuples in $K$ agree on exactly $d-1$ coordinates and pairwise differ in exactly one coordinate, say the $j$th coordinate ($1 \leq j \leq d$). Now consider the $d$-tuple $y$, obtained from $v$ by changing its $j$th entry to $a$. Clearly $y$ has no repetitions, $y\neq v$, and since $K$ is maximal, $y \in V(K)$. But by the definition of $A_I$, $y \in A_I$. Hence, $|K\cap A_I| \geq 2$; a contradiction. This proves (1).

By (1) and since $|K \cap A_I| \leq 2$, we have either $|K \cap A_I| =0$ or $|K \cap A_I| = 2$. In both cases $\phi_{I}$ satisfies the equation in (ii), thus (ii) holds.

Finally, part (iii) is a consequence of Lemma~\ref{IS}. First, we take a subset $\mathcal{A} \subseteq \mathcal{A}(n,d)$ of cardinality $n-1 \choose d$ which is independent in the sense of Lemma~\ref{IS}.
Next, we consider the functions in $\{\phi_{I}~|~A_I \in \mathcal{A} \}$.  The existence of unique $d$-tuples shows that for each $\phi_I$, there is a $d$-tuple $a$ such that $\phi_I(a)\ne0$, but $\phi_J(a)=0$ for every $A_J \in {\mathcal A} \setminus \{ A_I \}$.
Thus we have a linearly independent set in ${\mathcal F}_X$ of cardinality $ n-1 \choose d$.
\end{proof}

\section{Proof of the main result}
\label{sec:ProofOfMainResult}

Let $G$ be a group, $H \leq G$ a subgroup of $G$ and $T \subseteq G$. The \emph{Schreier coset graph} on $G / H$ generated by $T$ is the graph $X=X(G,H,T)$ with $V(X)=G/H=\{gH:~g\in G\}$, the set of left cosets of $H$, and there is an edge $(gH,tgH)$ for each coset $gH$ and each $t \in T$. If $T$ is inverse-closed, then $X$ is an undirected multigraph (possibly with loops). Note that if $1_G$ is the identity element of $G$, then $X(G,\{1_G\},T)=X(G,T)$ is the Cayley graph on $G$ generated by $T$. The following lemma is well-known, see, e.g.~\cite{Friedman00}.

\begin{lemma}
\label{eigenforeqi}
Let $G$ be a group, $T$ an inverse-closed subset of $G$, and $H \leq K \leq G$. If $\lambda \in \mathbb{R}$ is an eigenvalue of $X(G,K,T)$ of multiplicity $p$, then $\lambda$ is an eigenvalue of $X(G,H,T)$ of multiplicity at least $p$.
\end{lemma}

The following lemma is straightforward.

\begin{lemma}
\label{tran}
For $1\leq k \leq n-1$, let $S_{n-k}'$ be the subgroup of $S_n$ containing all permutations fixing the elements $n-k+1,\dots,n$.
For each $k$-tuple $(i_1,\dots,i_k) \in [n]^k$, where $i_1,\dots,i_k$ are pairwise distinct, let $\sigma_{(i_1,\dots,i_k)} \in S_n$ be a permutation so that $\sigma_{(i_1,\dots,i_k)}(n-j+1)=i_j$, for $j=1,\dots,k$.
Let $Z:=\{\sigma_{(i_1,\dots,i_k)}: i_1,\dots ,i_k ~\textrm{~are pairwise distinct} \}$. Then, $S_n$ is a disjoint union of all left cosets $zS_{n-k}'$, $z \in Z$.
\end{lemma}

Let $X[k]$ be the Schreier coset graph $X(S_n,S_{n-k}',T)$, where $1 \leq k \leq n-1$ and $T \subseteq S_n$ is inverse-closed. By Lemma~\ref{tran}, we see that $X[k]$ is isomorphic to the graph whose vertex set $V(X[k])$ is the set of all $k$-tuples (with no repetition) from the set $[n]$, in which a $k$-tuple $(i_1,\dots,i_k)$ is adjacent to $(t(i_1),\dots,t(i_k))$, for each $t\in T$. (This was also observed by Bacher in \cite{Ba}.) Note that elements of $T$ may give rise to multiple edges and loops in $X[k]$.

Now suppose $T=\{(1~2),(1~3),\dots,(1~n)\}$. Let $V_1 \subseteq V(X[k])$ be the set of $k$-tuples containing the value $1$. Let $\overline{V_1}:=V(X[k]) \setminus V_1$. The following lemma is easily verified.

\begin{lemma}
\label{propofeqi}
Let $V_1,\overline{V_1}$ be the partition of $X[k]$ defined above. Then the following holds:
\begin{itemize}
	\item[\rm(i)] No two distinct vertices of $\overline{V_1}$ are adjacent in $X[k]$.
	\item[\rm(ii)] Every $v \in \overline{V_1}$ has $n-k-1$ self-loops and $k$ neighbors in $V_1$.
	\item[\rm(iii)] Let $u \in V_1$. Then $u$ is adjacent in $X[k]$ to exactly $n-k$ vertices in  $\overline{V_1}$, that are obtained from $u$ be replacing the entry 1 with a number different from the $k$ entries of $u$.
\end{itemize}
\end{lemma}

\begin{lemma}
\label{gettingeigenvalues}
Let $1 \leq k \leq n-2$. Then $n-k-1$ is an eigenvalue of $X[k]$ with multiplicity at least ${n-2 \choose k}$.
\end{lemma}

\begin{proof}
Let $X[k]=(V[k],E[k])$ and let $V_1$ and $\overline{V_1}$ be as above. We construct the following auxiliary graph $\mathcal{G}$. The vertex set of $\mathcal{G}$ is $\overline{V_1}$, and two vertices $u,v \in \overline{V_1}$ are adjacent if and only if $v\neq u$ and there exists $w \in V_1$ such that $wu,wv \in E[k]$. Clearly, if $vu \in E(\mathcal{G})$ then $v$ and $u$ differ in exactly one coordinate. Hence, $\mathcal{G}$ is isomorphic to $P(k,n-1)$. Note that $n-1$ comes from the fact that the value $1$ does not appear in any of the vertices of $\overline{V_1}$. Observe that $k<n-1$, so $P(k,n-1)$ is defined.

Let $\mathcal{F}_{\mathcal{G}}$ be the set of functions $\phi: \overline{V_1} \to \{-1,0,1\}$ satisfying the properties stated in Lemma~\ref{function}, and let $\mathcal{F} \subseteq \mathcal{F}_{\mathcal{G}}$ be a linearly independent set in $\mathcal{F}_{\mathcal{G}}$ of cardinality $n-2 \choose k$, whose existence is given by Lemma~\ref{function}(iii).

For each $\phi \in \mathcal{F}$, let  $\phi' : V[k] \rightarrow \{0,1,-1\}$ be an extension of $\phi$ to $V[k]$ defined by:

$$
\phi'(v) = \left\{
              \begin{array}{rl}
              	\phi(v), & v \in \overline{V_1},\\[1mm]
                0, & v \in V_1.
              \end{array}
             \right.
$$
To conclude the proof, it suffices to show that for every $\phi \in \mathcal{F_G}$, $\phi'$ is an eigenvector of $X[k]$ with eigenvalue $n-k-1$, since $\mathcal{F_G}$ is an independent set.  This task is equivalent to verifying that, for every $v\in V[k]$, the following eigenvalue equation holds:
\begin{equation}
\label{eq:1}
\mbox{$(n-k-1)\phi'(v)=\displaystyle\sum_{vu\in E[k]} \phi'(u)$}
\end{equation}
where the sum is over all edges of $E[k]$ incident with $v$ including possible self-loops.
Recall that  $V[k]=V_1 \cup \overline{V_1}$.  If $v \in \overline{V_1}$, then by Lemma~\ref{propofeqi}~(i) and (ii), the only edges  contributing non-zero values to the right hand side of (\ref{eq:1}) are the $n-k-1$ self-loops at $v$. Hence (\ref{eq:1}) holds in this case.
If $v \in V_1$, then the left hand side of (\ref{eq:1}) is equal to 0. Now, by the definition of $\phi'$, the only edges contributing non-zero values to the right hand side of (\ref{eq:1}) are the edges $vu$, where $u \in \overline{V_1}$. By Lemma~\ref{propofeqi}~(iii), $v$ has precisely $n-k$ neighbors in $\overline{V_1}$, and they form a clique $K$ in $\mathcal{G}$ of cardinality $n-k$. By Lemma~\ref{basic-prop}, $K$ is a clique of maximum cardinality in $\mathcal{G}$ (recall that $\mathcal{G}$ is isomorphic to $P(k,n-1)$). By Lemma~\ref{function}~(ii),  the sum of the $\phi'$-values of the vertices of every maximum clique is zero and thus (\ref{eq:1}) holds in this case as well.
\end{proof}

The proof of the main result follows at once.

\begin{proof}[Proof of Theorem~\ref{main} (for non-zero eigenvalues).]
Since $G=X(S_n,T_n)$ is $(n-1)$-regular, $n-1$ is an eigenvalue of $G$ of multiplicity 1. Lemma~\ref{gettingeigenvalues} implies that for $1\leq k \leq n-2$, $n-k-1$ is an eigenvalue of $X[k]$ with multiplicity at least ${n-2 \choose k}$, and hence  by Lemma~\ref{eigenforeqi} also of $G$.  The same conclusion holds for the negative values since $G$ is bipartite, and hence the spectrum is symmetric with respect to 0. \end{proof}

\section{Eigenvalue zero}

In this section we prove that 0 is an eigenvalue of $X(S_n,T_n)$ of multiplicity at least $\binom{n-2}{2}$.

Let $K(2,n)$ be the graph on vertex set $S_{2,n}$ (all pairs of distinct elements from $[n]$). Two such pairs are adjacent if and only if either they have the same second coordinate, or they are transpose of each other (one is obtained from the other by interchanging the coordinates).

\begin{proposition}
\label{prop:cover}
The Cayley graph $X(S_n,T_n)$ is a cover over $K(2,n)$.
\end{proposition}

\begin{proof}
For $(i,j)\in S_{2,n}$, let $U_{ij} = \{\pi\in S_n\mid \pi(i)=1 \textrm{ and } \pi(j)=n\}$.
Consider the mapping $p:S_n\to S_{2,n}$ defined as $p(\pi)=(i,j)$ if $\pi\in U_{ij}$.
We claim that $p$ is a covering projection $X=X(S_n,T_n)\to K(2,n)$. Since both graphs are regular of degree $n-1$, it suffices to see that every $\pi\in U_{ij}$ is adjacent in $X$ to a permutation in  $U_{lj}$ for each $l\in [n]\setminus\{i,j\}$ and is adjacent to a permutation in $U_{ji}$.
This is confirmed below. Clearly, if $t=\pi(l)$, where $l\ne i,j$, then $t\ne 1,n$ and
$(1\, t)\pi(l)=1$ and $(1\, t)\pi(j)=n$. Thus $\pi' = (1\, t)\pi \in U_{lj}$. Similarly,
$(1\, n)\pi(j)=1$ and $(1\, n)\pi(i)=n$. Thus $\pi' = (1\, n)\pi \in U_{ji}$.
This completes the proof since in each case, $\pi'$ is a neighbor of $\pi$ in $X$.
\end{proof}

The following corollary provides the missing evidence for the completion of the proof of Theorem \ref{main}.

\begin{proposition}
\label{prop:0 eigenvalue}
If\/ $n\ge 4$, then zero is an eigenvalue of $K(2,n)$ and hence also of $X(S_n,T_n)$ of multiplicity at least $\binom{n-1}{2}$.
\end{proposition}

\begin{proof}
It is well known that eigenvalues of the base graph are also eigenvalues of the cover. By Proposition \ref{prop:cover}, it suffices to show that 0 is an eigenvalue of $K(2,n)$ with multiplicity at least $\binom{n-1}{2}$.

Let $A,B \subset [n]$ be disjoint subsets of $[n]$, where $|A|\ge2$ and $|B|\ge 2$. For $i\in [n]$, let $\alpha_i$ and $\beta_i$ be real numbers so that $\alpha_i\ne0$ for $i\in A$, $\alpha_i=0$ for $i\notin A$, $\beta_i\ne0$ for $i\in B$, and $\beta_i=0$ for $i\notin B$, such that $\sum_{i=1}^n \alpha_i = 0$ and $\sum_{j=1}^n \beta_j = 0$. Finally, for each $(i,j)\in S_{2,n}$, define
\begin{equation}
    x_{ij} = \alpha_i \beta_j + \alpha_j \beta_i. \label{eq:xij}
\end{equation}
Observe that $x_{ij}=x_{ji}$ and $x_{ii}=0$ for any $i,j\in [n]$.

We claim that $x=(x_{ij})$ is an eigenvector for eigenvalue 0 in $K(2,n)$.
To see this, we have to show that the sum $s$ of the values on all neighbors of $(i,j)$ in $K(2,n)$ is zero. But this is easy to see:
$$
   s = x_{ji} + \sum_{l\ne i,j} x_{lj} = \sum_{l=1}^n x_{lj} = \sum_{l=1}^n (\alpha_l \beta_j + \alpha_j \beta_l) = \beta_j\,\sum_{i=1}^n \alpha_l + \alpha_j\,\sum_{i=1}^n \beta_l = 0.
$$

The eigenvectors $(x_{ij})$ for eigenvalue 0 as defined above span a subspace of dimension at least $\binom{n-1}{2}$. The proof is by induction on $n$. For $n=4$, consider partitions $A\cup B$ of $\{1,2,3,4\}$, where $A=\{1,4\}$, $A=\{2,4\}$, or $A=\{3,4\}$ (respectively), and $B=[4]\setminus A$. They give three linearly independent vectors. To see this, note that each of the corresponding eigenvectors defined by (\ref{eq:xij}) has a non-zero value where the other two have value zero. For $n\ge 5$, consider $\binom{n-2}{2}$ independent vectors obtained by taking subsets $A,B$ of $[n-1]$. They all have coordinate 0 for every $(i,n)$ ($1\le i<n$). Finally, we can add $n-2$ other eigenvectors that have precisely one non-zero coordinate $(k,n)$ for $k\in\{1,\dots,n-2\}$: for the $k$th one, take $A=\{1,n\}$ and $B=\{k,n-1\}$, except for $k=1$, when $A=\{n-1,n\}$ and $B=\{1,n-1\}$). All together we have $\binom{n-2}{2} + n-2 = \binom{n-1}{2}$ independent eigenvectors.
\end{proof}





\bibliographystyle{abbrv}
\bibliography{refs}
\end{document}